\documentclass[12pt]{amsart}

\usepackage{graphicx}
\usepackage{anysize}
\marginsize{1in}{1in}{1in}{1in}

\usepackage{amsfonts,amssymb,latexsym,amsmath}
\usepackage{times}
\usepackage{mathrsfs}

\newcommand{\CC}{\mathbb{C}}
\newcommand{\DD}{\mathbb{D}}

\newcommand{\HH}{\mathbb{H}}
\newcommand{\LL}{\mathbb{L}}
\newcommand{\NN}{\mathbb{N}}

\newcommand{\PP}{\mathbb{P}}

\newcommand{\SSS}{\mathbb{S}}

\newcommand{\VV}{\mathbb{V}}

\newcommand{\ZZ}{\mathbb{Z}}

\newcommand{\cC}{\mathcal{C}}
\newcommand{\fF}{\mathcal{F}}
\newcommand{\gG}{\mathcal{G}}
\newcommand{\hH}{\mathcal{H}}

\newcommand{\oO}{\mathcal{O}}

\newcommand{\vV}{\mathcal{V}}

\newcommand{\Jbar}{\overline{J}}

\newtheorem{lemma}{Lemma}
\newtheorem{definition}[lemma]{Definition}
\newtheorem{proposition}[lemma]{Proposition}

\newtheorem{corollary}[lemma]{Corollary}
\newtheorem{theorem}[lemma]{Theorem}

\title{A support theorem for Hilbert schemes of planar curves}

\author{Luca Migliorini} \author{Vivek Shende}
\begin{document}

 \begin{abstract}
   Consider a family of integral complex
   locally planar curves whose relative Hilbert scheme of
   points is smooth.  
   The decomposition theorem of Beilinson, Bernstein, and Deligne
   asserts that the pushforward of the constant sheaf on the relative
   Hilbert scheme
   splits as a direct sum of shifted semisimple perverse sheaves.  We will
   show that no summand is supported in positive codimension.  
   It follows that the 
   perverse filtration on the cohomology of the compactified Jacobian
   of an integral plane curve encodes the cohomology of {\em all} Hilbert
   schemes of points on the curve.  Globally, it follows that
   a family of such curves
   with smooth relative compactified Jacobian
   (``moduli space of D-branes'') in an irreducible curve class on 
   a Calabi-Yau threefold will contribute equally to the
   BPS invariants in the formulation of
   Pandharipande and Thomas, and in the formulation of 
   Hosono, Saito, and Takahashi.
 \end{abstract}

\maketitle

\section{Introduction}

\noindent
In this note a {\em curve} will always be 
{\em integral, complete, locally planar, 
and defined over $\CC$.}\footnote{This
reflects a limitation of the authors rather than a certainty that
the methods do not work in characteristic $p$.}

\vspace{2mm}
Let $C$ be a curve of arithmetic genus $g$.  The Hilbert scheme
of points $C^{[d]}$ parameterizes length $d$ subschemes of $C$; it
is complete, integral, $d$-dimensional, and l.c.i. \cite{AIK,BGS}.
If $\pi: \mathcal{C} \to B$ is a family of curves,
there is a relative Hilbert scheme
$\pi^{[d]}:\mathcal{C}^{[d]} \to B$ with fibres $(\mathcal{C}^{[d]})_b = 
(\mathcal{C}_b)^{[d]}$.  
Planarity of the curves ensures the existence of families in which the
total space of $\mathcal{C}^{[d]}$ is smooth (\cite{S}), 
see Theorem \ref{thm:smoothness} below; 
ultimately this is a consequence
of the smoothness of Hilbert scheme of points on a surface.  
When $\mathcal{C}^{[d]}$ is smooth, 
the decomposition theorem of Beilinson, Bernstein and
Deligne \cite{BBD} applied to the proper map 
$\pi^{[d]}: \mathcal{C}^{[d]} \to B $
asserts that $\mathrm{R} \pi^{[d]}_*  \CC$ decomposes
as a direct sum of shifted 
intersection complexes associated to local systems on constructible
subsets of the base.

Let $\widetilde{\pi}: \widetilde{\cC} \to \widetilde{B}$ denote
the restriction of $\pi$ to the 
smooth locus.  The Hilbert schemes of a smooth curve 
are its symmetric products, and in particular the map $\widetilde{\pi}^{[d]}$ 
is smooth.  Thus
the summand of $\mathrm{R} \pi^{[d]}_*  \CC[d + \dim B]$ 
with support equal to $B$
is $\bigoplus \mathrm{IC}(B,\mathrm{R}^{d+i} \widetilde{\pi}^{[d]}_* \CC)[-i]$. 
As pointed out by Macdonald \cite{mcd}, 
the cohomology of the symmetric products is
expressed in terms of the cohomology of the curves by the formula
\begin{equation} \label{eq:macdonald}
  \mathrm{R}^i \widetilde{\pi}^{[d]}_* \CC = 
  \bigoplus_{k= 0}^{\lfloor i/2 \rfloor} \left( \mbox{$\bigwedge$}^{i-2k} 
  \mathrm{R}^1 \widetilde{\pi}_* \CC \right)(-k)  = 
  (\mathrm{R}^{2d-i}\widetilde{\pi}^{[d]}_* \CC)(d-i) \,\,\,\,\,\,
  \mbox{for $i \le d$}
\end{equation} 
Even given this expression, computing 
$\mathrm{IC}(B,\mathrm{R}^i \widetilde{\pi}^{[d]}_* \CC)$ is a
nontrivial matter, about which we say nothing here.  
But at least   $\mathrm{R} \pi^{[d]}_* \CC[d + \dim B]$ contains no other summands:

\begin{theorem}
 \label{thm:support}
  Let $\pi: \cC \to B$ be a family of integral plane curves, 
  and let $\widetilde{\pi}:\widetilde{\mathcal{C}} \to
  \widetilde{B}$ its restriction to the smooth locus. 
  If $\cC^{[d]}$ is smooth, then 
  \[
   \mathrm{R} \pi^{[d]}_* \CC[d + \dim B] =  \bigoplus_{i=-d}^{d}
  \mathrm{IC}(B, \mathrm{R}^{d+i} \widetilde{\pi}^{[d]}_* \CC)[-i].
  \]
\end{theorem}
From now on we will use the notation 
\[
{}^p \mathrm{R}^i \pi^{[d]}_* \CC[d + \dim B] := {}^p {\hH}^i \left( \mathrm{R} \pi^{[d]}_* \CC[d + \dim B] \right)
\]
for the perverse cohomology sheaves of $\mathrm{R} \pi^{[d]}_* \CC[d + \dim B]$.

\medskip
The central term of Equation \ref{eq:macdonald} can be reinterpreted in
terms of the family of Jacobians of the curves.  Indeed, taking 
$\widetilde{\pi}^J:J(\widetilde{\mathcal{C}}) \to \widetilde{B}$ 
to be the family of Jacobians over the smooth locus, then
there is a (canonical) identification of local systems
\begin{equation} \label{eq:jac}
\mathrm{R}^i\widetilde{\pi}^J_* \CC = \mbox{$\bigwedge$}^i 
(\mathrm{R}^1 \widetilde{\pi}_* \CC)
\end{equation}  
Consequently, 
\begin{equation} \label{eq:hilbjac}
  \mathrm{R}^i \widetilde{\pi}^{[d]}_* \CC = \bigoplus_k 
  (\mathrm{R}^{i-2k} \widetilde{\pi}^J_* \CC) (-k) 
  = 
  (\mathrm{R}^{2d-i}\widetilde{\pi}^{[d]}_* \CC)(d-i) \,\,\,\,\,\,
  \mbox{for $i \le d$}
\end{equation}  

It can be convenient to express Equations (\ref{eq:macdonald}), (\ref{eq:jac}),
and (\ref{eq:hilbjac}) in the following formula:
\begin{equation} \label{eq:macdonaldseries}
  \sum_{d=0}^\infty \sum_{i=0}^{2d} q^d
  \mathrm{R}^i \widetilde{\pi}^{[d]}_* \CC
  = \frac{ \sum\limits_{i=0}^{2g} 
    q^i \bigwedge ^i (\mathrm{R}^1 \widetilde{\pi}_* \CC)}
 {(1-q \CC)(1-q \CC(-1) )}
  = 
  \frac{\sum\limits_{i=0}^{2g} q^i \mathrm{R}^i \widetilde{\pi}^J_* \CC}{(1-q \CC)
    (1-q \CC(-1))}    
\end{equation}

The family of Jacobians can be extended 
over the singular locus of $\pi$ 
to the {\em compactified Jacobian} \cite{AK}, 
$\pi^J: \Jbar^d(\cC) \to B$, whose fibre $\Jbar^d(\cC)_b = 
\Jbar^d(\cC_b)$ parameterizes  rank one,
degree $d$
torsion free sheaves on $\cC$.\footnote{It also extends to the 
{\em generalized
Jacobian} $J(\cC)$ whose fibre $J(\cC)_b$ parameterizes line bundles on 
$\cC_b$; this is a commutative group scheme of dimension $g$ of which
the affine part is of dimension $\delta(\cC_b)$.  This is a subscheme
of the compactified Jacobian, and acts on it.  Such actions
are  central to Ng\^o's arguments, but play no role here. 
}  The map $\pi^J$ is proper, and for
Gorenstein curves there is an Abel-Jacobi
map $AJ: \cC^{[d]} \to \Jbar^d(\cC)$ taking a subscheme to the dual of its ideal 
sheaf.\footnote{In general, it is better to define the Abel-Jacobi
map from the Quot scheme of the dualizing sheaf, see \cite{AK}.}
For $d > 2g-2$, the map $AJ$ is a $\PP^{d-g}$ bundle, and  ${\mathrm R}\,AJ_*\CC=\oplus_{i=0}^{d-g} \CC[-2i]$; thus the statement 
in Theorem \ref{thm:support} is true
in this range
for the map $\pi^J$ as well.  Over a sufficiently small open set, 
$\pi$ admits a section $\sigma$
with image in the smooth locus of the curves; twisting by $\oO(\sigma)$ 
identifies the $\Jbar^d(\cC)$ for varying $d$ and so $\pi^J_* \CC$ does
not depend on $d$. In particular, we recover
the support theorem for the map $\pi^J$, a special case of Ngo's support theorem \cite{N}.
It can be shown \cite[Prop. 14]{S} 
that smoothness of
the relative compactified Jacobian implies smoothness of all relative Hilbert
schemes.  
Therefore taking IC sheaves in Equations (\ref{eq:macdonald})
and (\ref{eq:jac}) yields the following Corollary.

\begin{corollary} \label{cor:perverse}
  Let $\pi: \mathcal{C} \to B$ be a family of integral plane curves
  of arithmetic genus $g$.
  If the relative compactified Jacobian $\Jbar(\cC)$ is smooth, then:
  \[
  {}^p \mathrm{R}^{i-d} \pi^{[d]}_* \CC[d + \dim B] = 
  \bigoplus_{k=0}^{\lfloor i / 2 \rfloor}  
  {}^p \mathrm{R}^{i-g-2k} \pi^{J}_* \CC[g + \dim B](-k) \,\,\,\, 
  \mbox{ for $0 \le i \le d$}
  \] 
  (The ${}^p \mathrm{R}^{i-d}$ for $i > d$ are determined similarly by duality.)
\end{corollary}

This corollary has a consequence for the enumerative geometry of Calabi-Yau
three-folds, which we briefly sketch.  Gopakumar and Vafa argued in \cite{GV} 
that
the cohomology of the moduli space of $D$-branes (roughly speaking, 
semistable sheaves supported on curves) on a Calabi-Yau $Y$ 
should give rise to
{\em integer} ``BPS'' invariants, one for each genus and homology class
in $\mathrm{H}_2(Y,\ZZ)$,  which encode the Gromov-Witten invariants
of $Y$.  Hosono, Saito, and Takahashi \cite{HST} use intersection cohomology
and the tools of \cite{BBD} to give a precise formulation; however, 
their proposal  is known {\em not}
to give the desired BPS numbers in general \cite{BP}.  A different definition
of integer BPS invariants is given by Pandharipande and Thomas 
\cite{PT} using the closely related
spaces of ``stable pairs'', which for integral planar curves are just the
Hilbert schemes of points.  By the work of Behrend \cite{B}, the BPS invariants
are extracted by a weighted Euler characteristic of these spaces, the 
weighting function depending only on the singularities of the moduli space.
For BPS invariants associated to {\em irreducible} homology classes, it is
sensible to discuss the contribution of an individual curve in both theories;
if the moduli space of sheaves on $Y$ is smooth along the locus of sheaves
supported on a curve $C$, then the intersection cohomology considerations 
may be neglected in \cite{HST}, and likewise 
the weighting function of Behrend may be neglected in \cite{PT}.  In this case,
taking Euler characteristics in the Corollary yields the equality of the 
contributions of the curve $C$ to these two theories.   Such curves will 
certainly appear if $Y$ contains a Fano surface, and indeed the enumerative 
geometry of curves on surfaces is also illuminated by the stable pairs 
spaces \cite{KST, KT}.

The Hilbert schemes also appear in the conjecture
of Oblomkov, Rasmussen, and the second author 
\cite{OS, ORS}.  This relates the cohomology of the Hilbert schemes of points 
on a locally planar curve to the Khovanov-Rozansky HOMFLY homology 
of the links of its singularities.  
The HOMFLY homology is a vector
space carrying three gradings, and its Poincar\'e polynomial 
$\overline{\mathcal{P}}$ is written 
in the variables $a, q, t$.  When the
curve has a unique singularity,  the conjecture 
implies that, up to a normalization, the lowest coefficient of $a$ in 
$\overline{\mathcal{P}}$ is
$\sum q^{2n} \mathfrak{w}(C^{[n]})$, where $\mathfrak{w}$ is the
weight polynomial 
$\mathfrak{w}(X) := \sum_{i,j} t^i (-1)^{i+j} \dim \mathrm{Gr}_W^i 
\mathrm{H}^j_c(X)$.  The present work shows
that the series on the RHS may be extracted instead from the 
perverse filtration on the cohomology of the compactified Jacobian
of $C$.  The perverse filtration is surely no easier to compute
directly than the Hilbert scheme series, but at least
in special cases the cohomology of the compactified 
Jacobian (under its alias of affine Springer fibre \cite{L}) appears in 
the geometric construction of the spherical 
representations of the rational Cherednik algebra \cite{VV}.  The 
perverse filtration interacts well with these geometric constructions, 
and as a consequence the series above can sometimes be computed after passing to
other incarnations of the Cherednik algebra and its representations
which are more suitable for computations.  Details appear in \cite{ORS}. 

\vspace{2mm}
Theorem \ref{thm:support} is inspired by the support theorem of B. C. 
Ng\^o \cite{N}, and is a consequence of it when $d > 2g-2$.  Nonetheless
our proofs -- we give two -- do not logically depend on his work, 
and rely on the deformation theory results in \cite{S}; in particular,
the analogue of the crucial "$\delta-$ regularity" assumption in \cite{N}, 
is automatically satisfied in our case once the total space $\mathcal{C}^{[d]}$ 
is smooth, see  Corollary
\ref{cor:codim}.

\vspace{2mm} \noindent {\bf Acknowledgements.} 
Corollary \ref{cor:perverse} was conjectured during a discussion between 
the authors and Lothar G\"ottsche.  
We are indebted to Zhiwei Yun for the 
suggestion that the induction procedure of Section \ref{sec:versal} be 
categorified  and for suggesting the statement of Proposition \ref{prop:pervind}.  
We also thank Davesh Maulik, Alexei Oblomkov, and 
Richard Thomas for enlightening
conversations, and Sam Gunningham and Ben Webster for helpful comments (on
MathOverflow) on Proposition \ref{prop:weights}. 
A different approach to the main theorem can be found in
the work of Davesh Maulik and Zhiwei Yun \cite{MY}, 
who deduce it, under additional hypotheses but in arbitrary characteristic,
from the support theorem of Ng\^o.  

\vspace{2mm} \noindent 
{\bf Conventions.}  We follow \cite{BBD} in declaring 
$\fF \in \mathrm{D}^b_c(X)$ perverse when $\dim \mathrm{Supp} \hH^i(\fF) \le -i$,
and the same holds for the Verdier dual.  
That is, if $X$ is smooth and $n$ dimensional, $\CC[n]$ is perverse.
In arguments of a topological nature, we omit Tate twists.
As mentioned at the outset, all curves are integral and have singularities
of embedding dimension 2.  All families of curves will enjoy 
a smooth base.  For a curve $C$,  we denote by $\overline{C}$ its normalization, and 
write $\delta(C)$ for the difference
between its arithmetic and geometric genera, which we term the
 {\em cogenus}: $ \delta(C)=p_a(C)-p_a(\overline{C})$.

\section{Background on relative Hilbert schemes and versal deformations}
\label{hsvd}

\noindent The Hilbert schemes of points on integral planar curves are singular, 
but not hopelessly so: 

\begin{theorem} \cite{AIK,BGS}. 
  Let $C$ be a complete integral planar curve.  Then $C^{[d]}$ is
  integral, complete, $d$-dimensional, and locally a complete intersection.
\end{theorem}

We systematically employ versal deformations of curve singularities.  
We will always mean this in the sense of analytic spaces,
see \cite{GLS} for a thorough treatment.  The base of a versal deformation
of a plane curve singularity is smooth.  If
$\pi: \cC \to B$ is a family of curves, we say it is {\em locally versal} at
$b$ if it induces versal deformations of all the singularities of $\cC_b$. This last condition may be 
rephrased as follows (\cite{FGvS} Section A): letting $\overline{\VV}(\cC_b)$ be the product
of the versal deformations of the singularities of $\cC_b$, there is a natural 
tangent map $T_bB \to T_b\overline{\VV}(\cC_b)$  at $b$ coming from the local-to-global spectral sequence for 
first order deformations of $\cC_b$. The family is locally versal if this tangent map is surjective.
 Such families have
in particular the following properties:

\begin{theorem} \label{thm:versal}
  \cite{DH,T}.  Let 
  $\pi: \cC \to B$ be a family of curves.  The cogenus is an upper
  semicontinuous function on $B$.  Local versality is an open condition,
  and in a locally versal family the locus of curves of cogenus
  at least $\delta$ is equal to the closure of the locus of $\delta$-nodal 
  curves.  In particular,  the locus of curves of cogenus
  $\delta$ in a locally versal family has codimension $\delta$. 
\end{theorem}

Any curve singularity can be found on a rational curve; for an explicit 
construction see e.g. \cite{L}.  Moreover, if $\mathcal{C} \to B$ is a
family of curves, then locally near $b \in B$ one can find a different
family $\mathcal{C}' \to B$ such that $\mathcal{C}'_b$ is rational with
the same singularities as $\mathcal{C}_b$ and the two families induce
the same deformations of the singularities of the central fibre.  

\begin{proposition} \label{prop:splitting} \cite{FGvS}.
  The map from the base of a versal deformation of an integral locally planar
  curve to the product of the versal deformations of its singularities
  is a smooth surjection. 
\end{proposition}

\begin{corollary} \label{cor:splitting} 
  Let $\pi: \cC \to B$ be a family of curves.  Fix $b \in B$, and let
  $\overline{\cC_b}$ be the normalization of $\cC_b$.  Then there exists
  a neighborhood $b \in U \subset B$ and a family $\pi':\cC' \to U$ such that
  $\cC'_b$ is rational with the same singularities as $\cC_b$, and
  $\cC$ and $\cC'$ induce the same deformations of these
  singularities on $U$, and in particular have the same discriminant
  locus.  Moreover, on $U$, we have an equality of local systems
  $\mathrm{R}^1 \widetilde{\pi}_* \CC = 
  \mathrm{R}^1 \widetilde{\pi}'_* \CC 
  \bigoplus \mathrm{H}^1(\overline{\cC_b})$,
  the latter summand meaning the constant local system with the specified fibre.
\end{corollary}
\begin{proof}
  Let $C'$ be a rational curve with the same 
  singularities as $\cC_b$; let $\mathcal{C'} \to \VV(C')$ be a versal
  deformation of $C'$, and, as above, let $\overline{\VV}(\cC_b)$ be the product
  of the versal deformations of the singularities of $\cC_b$.  
  We identify  $\overline{\VV}(\cC_b)=\overline{\VV}(C')$, and, 
  by Proposition \ref{prop:splitting}, the map
  $\VV(C') \to \overline{\VV}(\cC_b)$ is a smooth surjection, 
  so we may choose, possibly after shrinking,
  a local section $\sigma$.  Pulling back, again possibly after shrinking $B$, the family $\cC' \to \VV(C')$  to $B$ 
via the composition 
  \[
B \to \VV(\cC_b) \to \overline{\VV}(\cC_b) \stackrel{\sigma}{\to} \VV(C')  
    \]
  we obtain a family of rational
  curves $\pi': \cC'_B \to B$.

  Shrink $U \subset B$ further so that 
  the inclusion $\cC_b \to \cC|_{U}$ is a homotopy equivalence, and let 
  $\tilde{b} \in U$ be a point with  smooth fibre $\cC_{\tilde{b}}$.  Let
  $\vV$ be the summand of $R^1 \widetilde{\pi}_* \CC$ whose fibre at
  $\tilde{b}$ is the kernel of the composition of the
  specialization map $\mathrm{H}^1(\cC_{\tilde{b}}) \to 
  \mathrm{H}^1(\cC_b)$ with the pullback to the normalization
  $\mathrm{H}^1(\cC_b) \to \mathrm{H}^1(\overline{\cC_b})$.  
  This is a symplectic summand, let $\vV^\perp$ be its orthogonal complement.
  As $\vV$ contains all vanishing cycles,
  the Picard-Lefschetz formula ensures $\vV^\perp$ has trivial monodromy 
  and thus extends extends to a trivial
  local system over $B$ with 
  fibre $\vV^\perp_b = \mathrm{H}^1(\overline{\cC}_b)$.  On the other hand, 
  $\vV$ depends only
  on the deformation of the singularities, which is the same in
  $\cC$ and $\cC'$. 
\end{proof}

To make use of such a replacement, it is necessary to know that
the relative Hilbert scheme $\mathcal{C}'^{[d]}$ is smooth if
$\mathcal{C}^{[d]}$ is.  This follows from results of the second author
on the smoothness of relative Hilbert schemes \cite{S}, which we
now review. Recall that in force of our conventions, the base $B$ 
of a family is always supposed 
to be smooth.

\begin{proposition} \label{prop:down} \cite[Prop. 14]{S}
  Let $\pi:\cC \to B$ be a family of curves.  If $\cC^{[d]}$ is smooth,
  then $\cC^{[n]}$ is smooth for any $n \le d$. 
\end{proposition}

 In particular, taking $n=1$, we have that $\cC$ must be smooth.

  \begin{theorem} \label{thm:smoothness}
    Let $\cC \to B$ be a family of curves.  
    For $b \in B$, let $I$ be the image of
    $T_b B$ in
    the product of the first-order deformations of the singularities
    of $\cC_b$.  Then: 
    \begin{enumerate}
    \item The smoothness of $\cC^{[d]}$ along
      $\cC^{[d]}_b$ depends only on $I$.
    \item If $\cC^{[d]}$ is smooth along $\cC^{[d]}_b$, then
      $\dim I \ge \min(d, \delta(\cC_b))$. 
    \item 
      If $\dim I \ge d$ and $I$ is general
      among such subspaces, $\cC^{[d]}$
      is smooth along $\cC^{[d]}_b$.
    \item  
      $\cC^{[d]}$
      is smooth along $\cC^{[d]}_b$ for all $d$ if and only if
      $I$ is transverse to the image of the ``equigeneric ideal''. 
      It 
      suffices for $I$ to be generic of dimension at least $\delta(\cC_b)$.
    \end{enumerate}
  \end{theorem}
  \begin{proof}
    For (1),
    take $z$ a subscheme of $\cC^{[d]}_b$ which decomposes as 
    $z = \coprod z_i$ into subschemes of lengths $d_i$ supported at points
    $c_i$.   Let $(\overline{\cC}_i, c_i) \to (\VV_i, 0)$ 
    be miniversal deformations
    of the curve singularities $(\cC_b, c_i)$ and 
    $(B,b) \to \prod (\VV_i, 0)$ a map along which the multi-germ
    $\coprod (\cC, c_i) \to (B,b)$ pulls back.  Then analytically 
    locally the germ $(\cC^{[d]}, [z])$ pulls back 
    from $\prod (\overline{\cC}_i^{[d_i]},[z_i])$ along the same map.  
    As the fibres of $(\overline{\cC}_i^{[d_i]}, [z_i]) \to (\VV_i, 0)$ are
    reduced of dimension $d_i$ by \cite{AIK,BGS} 
    and the total space is smooth by \cite[Prop. 17]{S}, 
    the smoothness
    of the pullback depends only on the image of 
    $T_b B$ in $\prod T_0 \VV_i$, which is well defined as the $\VV_i$ were taken
    miniversal.  The miniversal deformation of
    the germ of a 
    curve at a smooth point being trivial, only the singularities contribute. 

    To check (2), we may by
    (1) assume the map $T_b B \to I$ is an isomorphism and then
    identify locally $B$ with its image in some representative 
    $\overline{B}$ of $\prod (\VV_i, 0)$.  
    Shrink $\overline{B}$ until it can be written as
    $B \times \DD$ for some polydisc $\DD$; by openness of smoothness we may
    shrink $\DD$ further until $\cC^{[d]}|_{B \times \epsilon}$ is smooth for all 
    $\epsilon \in \DD$.  It is known \cite{DH,T} that the locus of nodal
    curves with $\delta(\cC_b)$ nodes in $\prod \VV_i$ is nonempty 
    (and of codimension $\delta(\cC_b)$);  
    choose $\epsilon$ so the slice 
    $B \times \epsilon$ contains such a point $p$ corresponding to such
    a curve.  
    If $d \le \delta$, there is a point 
    $z \in \cC_p^{[d]}$ be a point naming a subscheme
    supported at $d$ nodes.  The Zariski tangent space $T_z \cC_p^{[d]}$
    is $2d$ dimensional, so $\cC_p^{[d]}$ cannot be smoothed out over a base
    of dimension less than $d$. 

    For (3), again assume $B$ is embedded in $\overline{B} = \prod \VV_i$;
    now the situation is analytically locally smooth over that in 
    \cite[Thm. 19]{S}. Finally,
    (4) is stated in \cite{FGvS} for the compactified Jacobian; it follows for
    $\cC^{[d]}$ for $d \gg 0$ because this space fibres smoothly over the Jacobian, and
    for lower $d$ by Proposition \ref{prop:down}.  
  \end{proof}

\begin{corollary} \label{cor:codim}
  If $\cC \to B$ is a family of curves with $\cC^{[d]}$ smooth, then 
  for $\delta \le d$, the locus of curves with cogenus $\delta$ is
  of codimension at least $\delta$ in $B$. 
\end{corollary}
\begin{proof}
  Suppose not; let $B'$ be a generic $\delta - 1$ dimensional smooth
  subvariety of $B$, then the restriction $\cC^{[d]} \times_B B'$ is smooth
  and $B'$ intersects the locus of curves of cogenus $\delta$.  This
  contradicts (2) of Theorem \ref{thm:smoothness}. 
\end{proof}

\noindent {\bf Remark}. Corollary \ref{cor:codim} explains why we do not 
require a ``$\delta$-regularity'' assumption as in \cite{N} -- in the case
of Hilbert schemes and Jacobians, it follows from smoothness of the total space.

\section{Estimates}

The following is a variation on the ``Goresky-MacPherson inequality" of \cite{N}, Section 7.3.

\begin{lemma} \label{lem:GM}
  Let $\pi:X \to Y$ be a locally projective 
  morphism of smooth varieties with irreducible fibres of dimension $n$. 
  Then  
  \begin{enumerate}
  \item  
 $ {}^p \mathrm{R}^j \pi_* \CC [\dim X]= 0$   {\rm if}  $|j| > n,$ {\rm and} 
$  {}^p \mathrm{R}^{\pm n} \pi_* \CC [\dim X]= \CC [\dim Y] $.
\medskip
  \item
 $ \hH^i ({}^p \mathrm{R}^j \pi_* \CC [\dim X])  = 0$ {\rm for} $|j |\neq n$  {\rm and}    $i  \ge n - \dim Y - | j | $.
  
  \end{enumerate}
  In particular, every summand of $\mathrm{R} \pi_* \CC$ is supported
  on a subvariety of codimension $< n$. 
\end{lemma}

\begin{proof}
  The first statement follows immediately from the fact that the fibres are connected of dimension $n$.
  The estimate is symmetric in $j$ and, by relative hard Lefschetz,
$  {}^p \mathrm{R}^j \pi_* \CC [\dim X] \simeq  {}^p \mathrm{R}^{-j} \pi_* \CC [\dim X]$,
  thus
   we may assume $j \ge 0$. 
  We check at a point $y \in Y$, 
  where by \cite{BBD},   
  $\hH^i({}^p \mathrm{R}^j \pi_* \CC[\dim X])_y$
  is a summand of $\mathrm{H}^{i+j + \dim X}(X_y, \CC)$.  This vanishes for dimension
  reasons if $i+j+ \dim X=i+j+ \dim Y +n > 2 \dim X_y =2n$.
  Finally,
  as the fibres
  are irreducible,  $R^{2n} \pi_* \CC \simeq \CC$. This top dimensional cohomology 
  is already accounted for by the summand ${}^p \mathrm{R}^n \pi_* \CC[
  \dim X] = \CC[\dim Y]$
  and thus the vanishing for $j = i$ is ensured.  The final statement follows
  because a summand supported on a subvariety $Y'$ is the IC sheaf associated
  to some local system on an open subset of $Y'$ and consequently the stalk of the cohomology sheaf
  in degree $-\dim Y'$ is non zero on a general point of $Y'$; this is prohibited by the stated estimate when 
  $\dim Y - \dim Y'  \ge n$. 
\end{proof}

\begin{lemma} \label{lem:estimate}
  Let $\pi:\cC \to B$ be a family of curves such that $\cC^{[d]}$ is smooth. 
  Then  for $i > 0$, and for every $j$, the support of the sheaf
  $\hH^{i}(\mathrm{IC}(B,\mathrm{R}^j\widetilde{\pi}^{[d]}_* 
  \CC)[-\dim B])$ is contained in the locus of curves of cogenus
  $> i$. 
\end{lemma}
\begin{proof}
  We check at some point $b \in B$ and write $\delta$ for the cogenus
  of $\cC_b$.  By semicontinuity
  of cogenus, in some neighborhood all curves have cogenus $\le \delta$; we 
  shrink $B$ to this neighborhood and show that 
  $\hH^{i}(\mathrm{IC}(B,\mathrm{R}^j\widetilde{\pi}^{[d]}_* 
  \CC)[-\dim B]) = 0$ for all $i \ge \delta$.  
  Shrinking $B$ further if necessary, let $\pi':\cC'\to B$ be the family
  of curves constructed in Corollary \ref{cor:splitting}, which we
  recall has the property that $\cC'_b$ is rational, 
  $\mathrm{R}^1 \widetilde{\pi}_* \CC = 
{\mathrm  R^1} \widetilde{\pi}'_* \CC \oplus \mathrm{H}^1(\overline{\cC_b})$,
  and by  item (1) of Theorem \ref{thm:smoothness}, $\cC'^{[d]}$ is smooth.
  Taking exterior powers 
  and comparing with Equation (\ref{eq:macdonald}), we see 
  that $\mathrm{R}^j\widetilde{\pi}^{[d]}_* 
  \CC$ is a sum of $\mathrm{R}^{\le j}\widetilde{\pi}'^{[d]}_* 
  \CC$; it will therefore suffice to check the assertion for the 
  family $\cC'$. 

  Note $\delta$ is the common arithmetic genus of the fibres of $\pi'$.  
  From Macdonald's formula (\ref{eq:macdonald}), all summands of
  $\mathrm{R}^i \widetilde{\pi}'^{[d]}_* \CC$
  appear already as summands of 
  $\mathrm{R}^i \widetilde{\pi}'^{[\min(d,\delta)]}_* \CC$.  As
  $\cC'^{[\min(d,\delta)]}$ is smooth by Proposition \ref{prop:down},
  we may as well assume $d \le \delta$.   By relative hard
  Lefschetz, it suffices to check the assertion for $j \le d$.  
  But now $j \le d \le \delta \le i$, thus by the previous lemma, 
  we are done.
\end{proof}

\noindent {\bf Remark}. 
Being an IC sheaf ensures
that the above mentioned cohomology is supported on {\em some} subspace of
codimension $i+1$.  The force of the lemma is to show this subspace lies inside
the codimension $i+1$ locus of curves of cogenus $i+1$. 
Experimental evidence suggests that the support
is {\em much} smaller, and it would be  interesting to have a 
precise characterization. 

\begin{lemma} \label{lem:stillic}
  Let $\pi: \cC \to B$ be a family of curves, $B' \subset B$ a smooth closed
  subvariety, and $\pi': \cC' \to B'$ the restricted family.  Assume
  $\cC^{[d]}$ and $\cC'^{[d]}$ are smooth.
  Denote
  by $\widetilde{\pi}$ and $\widetilde{\pi}'$ the respective smooth loci
  of the maps. 
  Then   $\mathrm{IC}(B, \mathrm{R}^i \widetilde{\pi}^{[d]}_* \CC)
  |_{B'}[\dim B' - \dim B] = 
  \mathrm{IC}(B', \mathrm{R}^i \widetilde{\pi}'^{[d]}_* \CC )$.  
\end{lemma}
\begin{proof}
  By induction on the codimension of $B'$ in $B$, we are 
  reduced to proving the statement 
  for $B'$ a Cartier divisor in $B$. By \cite{BBD}, Cor. 4.1.12,  
  the complex 
  $K:= \mathrm{IC}(B, \mathrm{R}^i \widetilde{\pi}^{[d]}_* \CC)  |_{B'}[-1]$ 
  is a perverse sheaf.  By proper base change, $K$ is a summand of
  $ \mathrm{R} \pi'^{[d]}_* \CC[d + \dim B']$.  As  $\cC'^{[d]}$ is smooth,
  $K$ must be the sum of IC complexes, and by Corollary \ref{cor:codim}
  the locus of curves of cogenus $\delta \le d$ appears in codimension at least
  $\delta$ in $B'$.  By Lemma \ref{lem:estimate} and the fact that
  the fibre is $d$-dimensional, $\dim\, \mathrm{Supp}\, \hH^i(K) < -i$ for
  $i \ne - \dim B'$.  Therefore no summand of $K$ is an IC complex associated
  to a local system supported in positive codimension in $B'$,
  and the claimed isomorphism follows from the obvious fact 
  that, on the smooth locus, $K$ coincides with the (shifted) local system
  $\mathrm{R}^i \widetilde{\pi}'^{[d]}_* \CC[\dim B']$. 
\end{proof}

\begin{corollary} \label{cor:subfamilies}
  Let $\pi:\cC \to B$ be a family of curves, and 
  $\pi':\cC' \to B'$ its restriction to a smooth subvariety of the base;
  assume $\cC^{[d]}$ and $\cC'^{[d]}$ are smooth.  Let $\fF$ be 
  the summand of 
  $\mathrm{R} \pi_*^{[d]} \CC[d+\dim B]$ not supported on all of $B$, i.e. 
  $ \mathrm{R} \pi_*^{[d]} \CC[d+\dim B]=   \left(    \bigoplus_{i=-d}^{d}
  \mathrm{IC}(B, \mathrm{R}^{d+i} \widetilde{\pi}^{[d]}_* \CC)[-i] \right) \bigoplus \fF $, 
    and similarly
  $\fF'$ for $B'$.  If $B' \not\subset \mathrm{Supp}\, \fF$, 
  then $\fF' = \fF|_{B'} [\dim B' - \dim B]$. 
\end{corollary} 

\section{Proof via reduction to rational curves}

\begin{proposition} \label{prop:smalld}
  Let $\pi:\cC \to B$ be a family of curves.
  Then, in some neighborhood of $b \in B$, 
  Theorem \ref{thm:support} holds for $d \le \delta(\cC_b)$. 
\end{proposition}
\begin{proof}
  Suppose not; let $\cC \to B$ be a counterexample over a base of minimal
  dimension.  Let $b \in B$ be any point in the support of a summand 
  $\fF$ of 
  $\mathrm{R} \pi_*^{[d]} \CC$ not supported on all of $B$. 
  By Theorem \ref{thm:smoothness} and Corollary \ref{cor:subfamilies},
  the restriction of the family to a general slice of dimension
  $d$ passing through $b$ remains a counterexample.  Therefore
  $d \ge \dim B$.  On the other hand, by (2) of
  Theorem \ref{thm:smoothness}, and the assumption that $d \le \delta(\cC_b)$;
  we must have $d \le \dim B$. 
  By Lemma \ref{lem:GM}, the support of $\fF$ is of codimension
  $< d$, thus it intersects a general $d - 1$ dimensional
  slice of $B$.  Again by Corollary \ref{cor:subfamilies}, the restricted
  family remains a counterexample, contradicting the assumption of minimal
  dimensionality.
\end{proof}

Now let $\pi: \cC \to B$ be a family of curves; 
shrinking to a neighborhood of
some $b \in B$,  let $\pi':\cC' \to B$ be the replacement family of 
Corollary \ref{cor:splitting}.  Then from Equation (\ref{eq:macdonaldseries}), 
we see
\begin{equation}
  \sum_{d=0}^\infty \sum_{i=0}^{2d} q^d
  \mathrm{R}^i \widetilde{\pi}^{[d]}_* \CC
  = 
   \left( 
     \sum_{d=0}^\infty \sum_{i=0}^{2d} q^d \mathrm{H}^i(\overline{\cC_b}^{[d]})
   \right) \otimes 
   \left(\sum_{i=0}^{2\delta(\cC_b)} q^i \mbox{$\bigwedge$}^i \mathrm{R}^1
   \widetilde{\pi}'_* \CC \right)
\end{equation}
As the final term is manifestly symmetric about $q^\delta$, 
the  series is determined by its first $\delta$ terms.  

To finish the proof of Theorem \ref{thm:support}, 
it would suffice to show that 
\begin{equation}\label{eq:notproduct}
\sum q^d \mathrm{H}^*(\cC_b^{[d]}) =
\left(\sum q^d \mathrm{H}^*(\overline{\cC_b}^{[d]}) \right) \mathcal{Z}_C(q)
\end{equation}
for a generating polynomial of vector spaces $\mathcal{Z}_C(q)$ of 
degree $2\delta$ with coefficients symmetric around $q^\delta$.  
Indeed, then the fibre at $b$ of both sides of the equality asserted  in
Theorem \ref{thm:support} would be determined in the same way by their values
for $C^{[\le \delta]}$, which by 
Proposition \ref{prop:smalld} are equal. 

However, we know no direct way to establish Equation \ref{eq:notproduct}, 
although of course it will follow as a 
consequence of Theorem \ref{thm:support}.   Instead, we prove
the product formula and check the symmetry 
in the Grothendieck group of varieties, in which we denote by 
$\LL$ the class of the affine line.  This is still sufficient,
because the {\em weight polynomial} both factors through the 
Grothendieck group of varieties and serves to witness 
the non-existence of summands of $\mathrm{R}\pi_*^{[d]} \CC[\dim B]$.
For $K$ a complex of vector spaces carrying a weight filtration, we
write the weight polynomial 
$\mathfrak{w}(K) := \sum_{i,j} t^i (-1)^{i+j} \dim \mathrm{Gr}_W^i 
\mathrm{H}^j(K)$.  For a variety $Z$, we abbreviate 
$\mathfrak{w}(Z)$ for $\mathfrak{w}(\mathrm{H}^*_c(Z))$.

\begin{proposition} \label{prop:weights} Suppose given 
a proper map $f:X \to Y$ between smooth varieties, and some summand
$\fF$ of $\mathrm{R} f_* \CC[\dim X]$.  If, for all 
$y \in Y$, we have  $\mathfrak{w}(\fF_y[- \dim X]) = \mathfrak{w}(X_y)$, then
$\fF = \mathrm{R} \pi_* \CC[\dim X]$.
\end{proposition}
\begin{proof}
 Let $\mathrm{R} f_* \CC[\dim X] = \fF \bigoplus \gG$; we must show
 that if $\mathfrak{w}(\gG_y) = 0$ for all $y \in Y$, then $\gG = 0$. 
 $\gG$ is a direct sum of shifted complexes 
 of the form  $\mathrm{IC}(L_i)$, with $L_i$ local systems
 supported on locally closed subsets of $B$ 
 underlying pure variations of Hodge structures. Then for $y$ a general
 point of the support, the purity of $\gG$ and the vanishing of 
 the weight polynomial force the vanishing of the local systems.
\end{proof}

Let $C$ be a curve, $C^{sm}$ its smooth locus, 
and $\overline{C}$ its normalization.  For
$p \in C$, we write $(C,p)^{[n]}$ for the subvariety of $C^{[n]}$ 
parameterizing subschemes set-theoretically supported at $p$; our notation
is meant to recall that it depends only on the germ of $C$ at $p$. 
Let $b(p)$ be the number of analytic local branches of $C$ near $p$. 
Splitting subschemes according to their support gives
the following equality in the Grothendieck group of varieties:
\begin{eqnarray}
\sum q^n [C^{[n]}] & = &
\sum q^n [(C^{sm})^{[n]}] \prod_{p \in C \setminus C^{sm}} \sum q^n [(C,p)^{[n]}] \\
\label{eq:normalized}
& = & \left(\sum q^n [\overline{C}^{[n]}] \right)
\left(
\prod_{p \in C \setminus C^{sm}} (1-q)^{b(p)} \sum q^n [(C,p)^{[n]}] \right)
\end{eqnarray}

This is the desired product formula. It remains to show that the final
term of Equation \ref{eq:normalized} is symmetric around $q^\delta$.  
After passing to Euler characteristics, this is shown in \cite{PT} 
using Serre duality; the argument below is similar.

\begin{proposition}  
\label{prop:bps}
Let $C$ be a Gorenstein curve of cogenus $\delta$,
with smooth locus $C^{sm}$ and $b(p)$ analytic local branches at a point
$p \in C$.  Define 
\[
Z_C(q) := \prod_{p \in C \setminus C^{sm}} (1-q)^{b(p)} \sum q^n [(C,p)^{[n]}] \]
Then $Z_C(q)$ is a polynomial in $q$ of degree $2 \delta$.  Moreover, 
writing $\LL$ for the class of the affine line, we have 
$Z_C(q) = (q^2 \LL )^\delta Z_C(1/q\LL)$. 
\end{proposition}
\begin{proof}
  By Equation \ref{eq:normalized}, we 
  may assume $C$ is a rational curve of arithmetic genus $g$; note in
  this case $Z(C) = (1-q)(1-q\LL) \sum q^d [C^{[d]}]$. 
  Fix a degree 1 line bundle 
  $\oO(1)$ on $C$.  We map $C^{[d]} \to \overline{J}^0(C)$
  by associating the ideal $I \subset \oO_C$ to the
  sheaf $I^* = \hH om(I,\oO_C) \otimes \oO(-d)$; the fibre is 
  $\PP(H^0(C,I^*))$.  For $\fF$ a rank one degree zero torsion
  free sheaf, we write the Hilbert function as 
  $h_\fF(d) = \dim \mathrm{H}^0(C,\fF \otimes \oO(d))$.  Then since over
  the strata with constant Hilbert function, the map from the Hilbert schemes
  to the compactified Jacobian is the projectivization of a vector bundle, 
  we have the equality
  $\sum q^d [C^{[d]}]  = \sum_h [\{\fF\,|\,h_{\fF} = h\}] \sum q^d 
  [\PP^{h(d)-1}]$.
  
  Fix $h = h_\fF$ for some $\fF$.  Evidently $h$ 
  is supported in $[0,\infty)$, and by Riemann-Roch and
  Serre duality is equal to
  $d+1-g$ in $(2g-2,\infty)$.  Inside $[0,2g-2]$, it
  either increases by 0 or 1 at each step.  Let $\phi_\pm(h)
  = \{ d\,|\, 2 h(d-1) - h(d-2) - h(d) = \pm 1\}$; evidently
  $\phi_- \subset [0,2g]$ and $\phi_+ \subset [1,2g-1]$, and
  \[Z_h(q):=(1-q) (1-q\LL) \sum q^d [\PP^{h(d)-1}] = 
  \sum_{d \in \phi_-(\fF)} q^d \LL^{h(d)-1} -
  \sum_{d \in \phi_+(\fF)} 
  q^{d} \LL^{h(d-1)} \]
  This is a polynomial in $q$ of degree at most $2g$, hence so is $Z_C(q)$. 
  
  Now let $\gG = \fF^* \otimes \omega_C \otimes \oO(2-2g)$, and $h^\vee = h_\gG$. 
  By Serre duality and Riemann-Roch, $h^\vee(d) = h(2g-2-d) + d + 1 - g$, so
  in particular, $d \in \phi_{\pm}(h^\vee) \iff 2g-d \in \phi_{\pm}(h)$.  It
  follows that $q^{2g} \LL^g  Z_h(1/q\LL) = Z_{h^{\vee}}(q)$.  As 
  $Z_C(q) = \sum_h  [\{\fF\,|\,h_{\fF} = h\}] Z_h(q)$, we obtain the 
  final stated equality. 
\end{proof}

This completes the (first) proof of Theorem \ref{thm:support}.

\section{Proof by reduction to nodal curves} \label{sec:versal}

\begin{lemma} \label{lem:reduction}
  If Theorem \ref{thm:support} holds for all versal families of curves, 
  then it holds for all families. 
\end{lemma}
\begin{proof}
  By Corollary \ref{cor:subfamilies}, the hypothesis implies 
  that Theorem \ref{thm:support} holds for any subfamily of a versal family.
  Now let $ \pi: \cC \to B$ be a family such that the theorem fails; 
  let $\fF$ be the summand 
  of $\mathrm{R} \pi^{[d]}_* \CC$ whose support is not all of $B$, and let
  $b \in B$ be a point such that $\fF_b \ne 0$. 
  Let $\phi: B \to \VV(\cC_b)$ be a map to the miniversal deformation, 
  and let $B' \subset B$ be a smooth closed subvariety through $b$ such that
  $\mathrm{d} \phi_b |_{B'}$ is an isomorphism onto the image of 
  $\mathrm{d} \phi_b (T_b B)$.  By item (1) of Theorem
  \ref{thm:smoothness}, 
  $\cC^{[d]}|_{B'}$ is still smooth.  According to 
  Corollary \ref{cor:subfamilies}, choosing 
  $B'\not\subset \mathrm{Supp}\,\fF$ 
  ensures that the restricted family
  still provides a counterexample in any neighborhood of $b$. 
  Shrinking still further, the map $B' \to \VV(\cC_b)$ may be taken
  to be the embedding of a smooth subvariety, giving a contradiction.
\end{proof}

We now prove Theorem \ref{thm:support} for the versal family.  The argument
is an induction on the cogenus, which depends crucially on the properties
of the versal family identified in Theorems \ref{thm:versal} and
\ref{thm:smoothness}. 
For clarity, we separate topological generalities from the specific
properties of the versal family.

\begin{definition}
  Let $X$ be a smooth complex analytic space with a constructible stratification
  $X = \coprod X_i$ such that $X_i$ is everywhere 
  of codimension $\ge i$.
  We write $\mathfrak{N}(\coprod X_i)$ for the full subcategory of 
  $\mathrm{D}^b_c(X)$ whose objects $\fF$ have the following property.  
 
  \begin{quote} 
  For $x \in X_i$ with $i< \dim X$ and
  for generic, sufficiently small,
  polydiscs
  $X \supset 
  \DD^i \times \DD \supset \DD^i \times 0 \ni x$,
  for sufficiently small
  $\epsilon \in \DD$, the restriction 
  \[\fF_x = \mathrm{R}\Gamma(\DD^i \times 0, \fF|_{\DD^i \times 0}) = 
  \mathrm{R}\Gamma(\DD^i \times \DD, \fF|_{\DD^i \times \DD}) \to 
  \mathrm{R}\Gamma(\DD^i \times \epsilon,\mathcal{F})\] 
  is an 
  isomorphism.  
  \end{quote}
\end{definition}

\begin{lemma} \label{lem:thick}
  $\mathfrak{N}(\coprod X_i)$
  is a thick triangulated subcategory of
  $\mathrm{D}^b_c(X)$, 
  i.e., it is closed
  under shifts, triangles, and taking summands. 
\end{lemma}

\begin{lemma} \label{lem:measure}
  Let $X^+ \subset X$ be an open subset such that 
  $X_i \setminus X^+$ is of codimension $> i$. 
  Then the composition 
   $\mathfrak{N}(\coprod X_i) \to \mathrm{D}^b_c(X) \to \mathrm{D}^b_c(X^+)$, where the second functor is given by restriction to the open set $X^+$, is faithful.
\end{lemma}

\begin{proof}

  Note that the condition on $X^+$ implies that $X_{\dim X} \subseteq X^+$.
  Consider $\mathcal{F} \in \mathfrak{N}(\coprod X_i)$ such that
  $\mathcal{F}|_{X^+} = 0$.  We must show $\mathcal {F}_x = 0$ for all $x \in X$.  
  Suppose by induction $\mathcal{F}_x = 0$ for $x \in X_{< i}$ and consider 
  $x \in X_i \setminus X^+$.  
  Evidently $(X_i \setminus X^+) \cup X_{> i}$ is of codimension $> i$, 
  so the generic $\DD^i \times \epsilon$ from the definition of 
  $\mathfrak{N}( \coprod X_i)$ passing near $x$ misses this locus
  completely.  Thus by assumption and the induction hypothesis, 
  $\fF_x = \mathrm{R}\Gamma(\DD^i \times \epsilon,\mathcal{F}) = 0$. 
 
\end{proof}

We now apply the statements above to $X=B$, the base of a locally versal family, with  the stratification 
$B= \coprod B_i$ given by the loci of curves of cogenus $i$,  
 and $X^+\subseteq B$ the open set 
parameterizing curves with at worst nodal singularities.

\begin{proposition}  \label{prop:pushcinn}
  Let $\pi: \cC \to B$ be a locally versal family of curves.  
  Let $B_i$ be the locus of curves of cogenus $i$.
  Then $\mathrm{R}\pi^{[d]}_* \CC[\dim B] \in \mathfrak{N}(\coprod B_i)$.
\end{proposition}
\begin{proof}
  We check at some $b \in B_\delta$.  
  The definition of $\mathfrak{N}$ is local on the base; 
  as $\widetilde{\pi}$ is proper, after shrinking 
  $B$ the inclusion $\cC^{[d]}_b \hookrightarrow \cC^{[d]}$ becomes a homotopy
  equivalence. 
  Any sufficiently small polydisc $b \in \DD^{\delta} \times \DD \subset 
  B$ will induce homotopy equivalences
  $\cC_b^{[d]}  \to \cC^{[d]}|_{\DD^{\delta} \times 0} \to 
  \cC^{[d]}|_{\DD^{\delta}\times \DD}$.  By item (3) of Theorem \ref{thm:smoothness},
  a generic choice  ensures
  that the latter two spaces are smooth, possibly after further shrinking the
  discs; by openness of smoothness we can shrink $\DD$ still further so that
  the projection $\cC^{[d]}_{\DD^\delta \times \DD} \to \DD$ is smooth.  
  It follows that, possibly after shrinking $\DD^\delta$ further,  
  $\mathrm{H^*}(\cC^{[d]}_{\DD^\delta \times \DD}) = 
  \mathrm{R}\Gamma(\DD^\delta \times \DD,\mathrm{R}\pi^{[d]}_* \CC ) \to 
  \mathrm{R}\Gamma(\DD^\delta \times \epsilon,\mathrm{R}\pi^{[d]}_* \CC)
  = \mathrm{H^*}(\cC^{[d]}_{\DD^\delta \times \epsilon})$ is an isomorphism
  (see Proof of Thm. 1 in \cite{FGvS} for a similar argument).
\end{proof}

\begin{proposition} \label{prop:support}
   Theorem \ref{thm:support} holds for all locally 
   versal families of curves.
\end{proposition}
\begin{proof}
  Let $\pi: \cC \to B$ be a locally versal family of curves, and
  let $B_i$ be the locus of curves of cogenus $i$.  
  Let $\fF$ be any summand of $\mathrm{R}\pi_*^{[d]} \CC$ supported
  on a proper subvariety of $B$.  Then by Lemma \ref{lem:thick}
  and Proposition \ref{prop:pushcinn}, $\fF \in \mathfrak{N}(\coprod B_i)$. 
  By Theorem \ref{thm:versal} \cite{DH,T}, 
  the locus of nodal curves is dense in each $B_i$; thus by
  Lemma \ref{lem:measure} we need only check that the restriction of
  $\fF$ to the locus of nodal curves is zero, i.e., that Theorem
  \ref{thm:support} holds for families of nodal curves; 
  this is done in the next Lemma. 
\end{proof}

\begin{lemma} \label{lem:piclef}
  Theorem \ref{thm:support} holds for  locally 
  versal families of nodal curves.
\end{lemma}
\begin{proof}
  Let $\pi: \cC \to B$ be such a family. 
  Let $b \in B$ be the base point, let $\{c_1, \cdots c_{\delta}\} 
  \subseteq \cC_b$ be the nodal set of the central curve $\cC_b$, 
  and denote by $r$ its geometric genus. 
  Shrink $B$ if necessary, we can assume:
  \begin{enumerate}
  \item 
    the discriminant locus is a normal crossing divisor $\Delta= \cup D_i$ 
    with 
    $i=1, \cdots, \delta$, where $D_i$ is the locus in which the $i-$th node $c_i$ is preserved.
  \item
    If  $b_0$ is  such that $\cC_{b_0}$ is nonsingular, the vanishing cycles
    $\{ \zeta_1, \cdots, \zeta_{\delta} \}$ in $\cC_{b_0}$ associated with 
    the nodes of $\cC_b$ 
    are disjoint.
\end{enumerate}
As the curve  $\cC_b$ is irreducible, the cohomology classes in 
$\mathrm{H}^1(\cC_{b_0})$ of these vanishing cycles are linearly
independent, and can then be completed to a symplectic basis. 

Let $T_i$ be the generators of the (abelian) local fundamental group $\pi_1(B \setminus \Delta, b_0)$ where $T_i$ 
corresponds to ``going around $D_i$''. Then the monodromy defining the local system ${\mathrm R}^1 \widetilde{\pi}_* \CC$  on $B \setminus \Delta$   is given via the Picard-Lefschetz formula,
and, in the symplectic basis above, has a Jordan form consisting of $\delta$ Jordan blocks of length 2. From this it is easy to 
compute the invariants of the local systems obtained applying any linear algebra construction to $\mathrm{R}^1 \widetilde{\pi}_* \CC$, such as those 
who appear in $\mathrm{R}^i \widetilde{\pi}^{[d]}_* \CC$. Let $\SSS^{i,[d]}$ 
be the linear algebra operation, described by Formula \ref{eq:macdonald}, 
such that 
$\mathrm{R}^i \widetilde{\pi}^{[d]}_* \CC=\SSS^{i,[d]} R^1 \pi_* \CC $. 
Denote by $j: B \setminus \Delta \to B$ the open inclusion.

We have a natural isomorphism  
\[
 \label{eq:natiso}
 \left( \SSS^{i,[d]} \mathrm{H}^1(\cC_{b_0}) \right)^{\pi_1(B \setminus \Delta, b_0)} = 
 \hH^{-{\rm dim}B}(\mathrm{IC}(B, \mathrm{R}^i \widetilde{\pi}^{[d]}_* \CC) )_b
\]
between the monodromy invariants on $\SSS^{i,[d]} \mathrm{H}^1(\cC_{b_0})$
and the stalk at $b$ of the first non-vanishing cohomology sheaf 
of the intersection cohomology complex of 
$\mathrm{R}^i \widetilde{\pi}^{[d]}_* \CC$.
The decomposition theorem in \cite{BBD} then implies that 
$\mathrm{H}^*(\cC^{[d]}_{b})$ contains the Hodge structure 
\[  
  \HH^{[d]} :=  \bigoplus_i \left( \SSS^{i,[d]} \mathrm{H}^1(\cC_{b_0}) 
  \right)^{\pi_1(B \setminus \Delta, b_0)}
  \] 
as a direct summand, with the weight filtration defined in the standard way by the logarithms of the monodromy operators (see \cite{ck}).

It is easy to compute $\HH^{[d]}$ explicitly; presumably 
$\mathrm{H}^*(\cC^{[d]}_{b})$ can be computed by elementary methods and
shown to match; this would complete the proof.  In the absence of such
a calculation, we use Proposition \ref{prop:weights} and instead compare
weight polynomials.  On the one hand, we compute
$\sum q^d \mathfrak{w}(\HH^{[d]}) = (1- q + t^2 q^2)^{\delta} (1+tq)^{2 r}/
(1-q)(1-t^2 q)$.

On the other hand, 
when $C = \PP^1_+$ is a rational curve with a single node, Riemann-Roch ensures
that the Abel map is a projective bundle for any $d \ge 1$; when $d = 1$ 
we have $[\overline{J}^0(\PP^1_+)] = [\PP^1_+] = \LL$.  Thus we get the formula 
$\sum q^d [(\PP^1_+)^{[d]}] 
= (1 - q + q^2 \LL)/((1-q)(1-q\LL))$.  Comparison with Equation 
(\ref{eq:normalized})
gives $\sum q^d [\cC_b^{[d]}] = \left(\sum q^d [\overline{\cC_b}^{[d]}]\right)
(1 - q + q^2 \LL)^{\delta}$; taking weight polynomials gives the desired result.
\end{proof}

This completes the (second) proof of Theorem \ref{thm:support}.

\section{An application}
Given a projective map $F:X \to Y$ with $X$ nonsingular, the decomposition theorem (\cite{BBD})
$ 
   \mathrm{R} F_* \CC[ \dim X] =  \bigoplus_{i=-r}^{r}
  {}^p \mathrm{R}^i F_* \CC[\dim X][-i]
$
 gives, by proper base-change, an isomorphism
\[
\mathrm{H}^k(F^{-1}(y))\stackrel{\Phi_F}{\longleftarrow} \bigoplus_{i=-r}^{r}
 {\mathcal H}^{k-\dim X -i} ({}^p \mathrm{R}^i F_* \CC[\dim X])_y,
\]
for every $y \in Y$ and $k\in \NN$.
The filtration 
\[
\mathrm{H}^k_{\leq 0}( F^{-1}(y))\subseteq  \mathrm{H}^k_{\leq 1}( F^{-1}(y))  \subseteq \cdots \subseteq \mathrm{H}^k_{\leq 2r-1}( F^{-1}(y)) \subseteq \mathrm{H}^k_{\leq 2r}( F^{-1}(y))=\mathrm{H}^k( F^{-1}(y)), 
\]
where 
\[  \mathrm{H}^k_{\leq l}(F^{-1}(y)):= {\Phi_{F}}(  \bigoplus_{i=-r}^{-r+l}
{\mathcal H}^{k-\dim X -i} ( {}^p \mathrm{R}^i F_* \CC[\dim X])_y   ),
\]
is called the {\em perverse filtration } on $\mathrm{H}^k(F^{-1}(y))$ associated with $F$. 

Let $C$ be a complete integral locally planar curve. A deformation 
$
C\subset \cC \stackrel{\pi}{\longrightarrow} B \ni b,
$ with $C=\pi^{-1}(b)$,
such that $\cC^{[d]}$ is nonsingular, defines, by what above, a perverse filtration
$
\mathrm{H}^k_{\leq 0}({C}^{[d]})\subseteq \mathrm{H}^k_{\leq 1}({C}^{[d]})\subseteq \cdots \subseteq \mathrm{H}^k_{\leq 2d-1}({C}^{[d]})\subseteq \mathrm{H}^k_{\leq 2d}({C}^{[d]})=
\mathrm{H}^k({C}^{[d]}) 
$
on the cohomoloy groups of the $d-$th Hilbert scheme of $C$.
The following Proposition shows that this filtration is in fact intrinsic, i.e. it does not depend on $\pi$:

\begin{proposition}
\label{prop:pervind}
Let $C$ be a complete integral curve, Let $C\subset \cC \stackrel{\pi}{\longrightarrow} B \ni b,$ and  $C\subset \cC' \stackrel{\pi '}{\longrightarrow} B \ni b'$ be two deformations of 
$C$ with the property that  $\cC^{[d]}$ and $\cC'^{[d]}$ are nonsingular.
Then the perverse filtrations on $\mathrm{H}^k({C}^{[d]})$ 
associated with ${\pi}$ and ${\pi}'$ coincide.
\end{proposition}

\begin{proof}
By Theorem \ref{thm:support} we have 
\begin{equation}
\label{perfiltrsupp} 
\mathrm{H}^k_{\leq l}({C}^{[d]}):= {\Phi_{\pi}}(  \bigoplus_{i=-d}^{-d+l}
{\mathcal H}^{k-d-\dim B-i} ( \mathrm{IC}(B, \mathrm{R}^{d+i} \widetilde{\pi}^{[d]}_* \CC))_b   ),
\end{equation}
and similarly for $\pi'$.
By appropriately shrinking $B$ and $B'$ around $b$ and $b'$ respectively, we may assume that the families are the pullback from a versal deformation 
$C\subset \cC_{\rm v} \stackrel{\pi_{\rm v} }{\longrightarrow} B_{\rm v} \ni b_{\rm v}$ of $C$.
By arguments analogous to those in Section \ref{hsvd} we are reduced to prove the statement in the case when  $\pi: \cC= \cC_{\rm v}\times_{B_{\rm v}}B \to B$ and  
$\pi': \cC'=\cC_{\rm v}\times_{B_{\rm v}}B' \to B'$ are the restriction 
of $\cC_{\rm v} {\longrightarrow} B_{\rm v}$
to two smooth subvarieties $B \stackrel{i}{\hookrightarrow}  B_{\rm v} \stackrel{i'}{\hookleftarrow} B'$  with $b_{\rm v}=b=b' \in B \cap B'$. The natural restriction maps 
\[
i_*\mathrm{R}\pi^{[d]}_* \CC =  i_*i^*   \mathrm{R}{\pi^{[d]}_{\rm v}}_* \CC  \longleftarrow
\mathrm{R}{\pi^{[d]}_{\rm v}}_* \CC \longrightarrow    i'_*i'^*   \mathrm{R}{\pi^{[d]}_{\rm v}}_* \CC =  i'_*\mathrm{R}\pi'^{[d]}_* \CC 
\]
induce isomorphisms on the cohomology sheaves at the point $b_{\rm v}$,
and the statement now follows immediately from the expression (\ref{perfiltrsupp}) applied to the maps $\pi, \pi'$ and $\pi_{\rm v}$, and Lemma  \ref{lem:stillic}.
\end{proof}

\bigskip
\bigskip
\bigskip
\bigskip

\begin{tabular*}{\textwidth}{@{\extracolsep{\fill}} l l } 
Luca Migliorini & Vivek Shende  \\
Dipartimento di Matematica & Department of Mathematics \\
Universit\`a di Bologna & Massachusetts Institute of Technology \\
Piazza Porta S.Donato, 5, & Building 2, Room 248 \\
40126 Bologna (ITALY) & 77 Massachusetts Avenue  \\
{\tt luca.migliorini@unibo.it} & Cambridge, MA 02139 (USA) \\
 & {\tt vivek@math.mit.edu} 
\end{tabular*}
 

\begin{thebibliography}{10}

\bibitem[AIK]{AIK} A. Altman, A. Iarrobino, and S. Kleiman, {\em
Irreducibility of the Compactified Jacobian}, Proceedings of the
Nordic Summer School N. A. V. F., (Symposium in Mathematics, Oslo, 
1976).

\bibitem[AK]{AK} 
A. Altman and S. Kleiman, {\em Compactifying the Picard Scheme}, 
Adv. in Math. {\bf 35} (1980), 50-112. 

\bibitem[B]{B} K. Behrend, {\em Donaldson-Thomas invariants via microlocal
geometry}, Annals of Math. {\bf 170} (2009), 1307-1338.



\bibitem[BBD]{BBD}
A.A. Beilinson, J. Bernstein, P. Deligne,  {\em Faisceaux pervers},
 Ast\`erisque {\bf 100}  (1982), 5-171. 

\bibitem[BGS]{BGS}
J. Brian\c con, M. Granger, and J.-P. Speder, 
{\em Sur le sch\'ema de Hilbert d'une courbe plane.}
Ann. sci. de l'\'Ecole Normale Sup\'erieure, S\'er. 4, 14 no. 1 (1981), 1-25.


\bibitem[BP]{BP} J. Bryan and R. Pandharipande, {\em BPS states in Calabi-Yau
3-folds}, Geometry \& Topology {\bf 5} (2001), 287-318. 


\bibitem[CK]{ck}
E. Cattani, A. Kaplan,
{\em Polarized mixed Hodge structures and the local monodromy of a variation of Hodge structure}  
Invent. Math. {\bf 67} (1982), 101-115.
 
\bibitem[DH]{DH} S. Diaz and J. Harris, {\em Ideals associated to 
deformations of singular plane curves}, Tr.
AMS {\bf 309} (1988), 433-468.



\bibitem[FGvS]{FGvS} B. Fantechi, L. G\"ottsche, and D. van Straten, 
{\em Euler number of the compactified Jacobian and multiplicity 
of rational curves}, J. Alg. Geom. {\bf 8}.1 (1999), 115-133. 


\bibitem[GLS]{GLS} G.-M. Greuel, C. Lossen, and E. Shustin, 
{\em Introduction to Singularities and Deformations}
(Springer, 2007). 
 

\bibitem[GV]{GV} R. Gopakumar and C. Vafa, 
{\em M-theory and Topological strings I \& II}, 
[hep-th/9809187] \& [hep-th/9812187].

\bibitem[HST]{HST} S. Hosono, M-H. Saito, A. Takahashi, {\em
Relative Lefschetz action and BPS state counting}, Int. Math. Res. Not.
{\bf 2001}.15 (2001), 765-782. 

\bibitem[KR05]{KR} M. Khovanov, L. Rozansky, {\em 
    Matrix factorizations and link homology}, I. [math.QA/0401268]
  and II. [math.QA/0505056].

\bibitem[KST]{KST} M. Kool, V. Shende, and R. Thomas, {\em A short proof
of the G\"ottsche conjecture} 
Geom. Topol. {\bf 15} (2011), 397-406.

\bibitem[KT]{KT} M. Kool and R. Thomas,
{\em Reduced classes and curve counting on surfaces I} [arXiv:1112.3069]
{\em and II} [arXiv:1112.3072]. 

\bibitem[L]{L} G. Laumon, {\em
 Fibres de Springer et jacobiennes compactifi\'ees}, [arxiv:0204109].

 \bibitem[M]{mcd} I. G. Macdonald,  
 {\em The Poincare Polynomial of a Symmetric Product}, 
 Math. Proc. of the Cambridge Philos. Soc. {\bf 58} (1962), 563-568.
  
\bibitem[MY]{MY} D. Maulik and Z. Yun, {\em Macdonald formula for  
curves with planar singularities}, [arxiv:1107.2175].

\bibitem[N]{N} B. C. Ng\^o, {\em Le lemme fondamental pour les alg\`ebres de Lie}, Publ. Math. de l'IHES {\bf 111}, (2010). 

\bibitem[ORS]{ORS} A. Oblomkov, J. Rasmussen, and V. Shende, 
{\em The Hilbert scheme 
of a plane curve singularity and the 
HOMFLY homology of its link}, [arxiv:1201.2115].

\bibitem[OS]{OS} A. Oblomkov and V. Shende, {\em The Hilbert scheme 
of a plane curve singularity and the 
HOMFLY polynomial of its link}, [arxiv:1003.1568].

\bibitem[PT]{PT}  R. Pandharipande; R. P. Thomas 
{\em Stable pairs and BPS invariants}
J. Amer. Math. Soc. {\bf 23} (2010), 267-297.
 
\bibitem[S]{S} V. Shende, {\em Hilbert schemes of points on a locally planar 
curve and the Severi strata of its versal deformation}, [arxiv:1009.0914], 
to appear   in Comp. Math.

\bibitem[T]{T} B. Tessier, {\em 
    R\'esolution simultan\'ee - I. Famille de courbes}, in 
  S\'eminaire sur les singularit\'es des surfaces, 
  Springer LNM {\bf 777} (1980). 

\bibitem[VV]{VV} M. Varagnolo and E. Vasserot, 
{\em Finite-dimensional representations of DAHA and affine Springer 
fibers: The spherical case}, Duke Math. J. {\bf 147}.3 (2009), 439-540.



\end{thebibliography}
\end{document}